\documentclass[11pt,a4paper,twoside]{article}
\usepackage[hmarginratio=1:1,bmargin=1.3in]{geometry}
\usepackage[OT2,T1]{fontenc}
\usepackage{amsmath, amsthm}
\usepackage{amsfonts}
\usepackage{amssymb}
\usepackage[all]{xy}
\usepackage{graphicx}
\usepackage{color}
\usepackage[nottoc, notlof, notlot]{tocbibind}
\usepackage{array}
\usepackage{url}
\usepackage{rotating}
\usepackage{fancyhdr}
\usepackage{fancyhdr}
\usepackage{titlesec}
\usepackage{xfrac} 
\usepackage{multicol}
\usepackage{enumitem}
\usepackage{ulem}

\titleformat{\section}[hang]{\center\Large\bf}{\thesection.}{0.5cm}{}

\DeclareSymbolFont{cyrletters}{OT2}{wncyr}{m}{n}
\DeclareMathSymbol{\Sha}{\mathalpha}{cyrletters}{"58}
\DeclareMathSymbol{\Brusse}{\mathalpha}{cyrletters}{"42}

\theoremstyle{plain}
\newtheorem{thm}{Theorem}[section]

\newtheorem{lem}[thm]{Lemma}
\newtheorem{pro}[thm]{Proposition}
\newtheorem{cor}[thm]{Corollary}

\theoremstyle{definition}
\newtheorem{defi}[thm]{Definition}
\newtheorem{rem}[thm]{Remark}

\newtheorem{exa}[thm]{Example}

\newcommand{\Z}{\mathbb{Z}}
\newcommand{\Q}{\mathbb{Q}}

\newcommand{\id}{\mathrm{id}}
\newcommand{\im}{\mathrm{Im}}

\newcommand{\aut}{\mathrm{Aut}}

\newcommand{\gm}{\mathbb{G}_{\mathrm{m}}}
\newcommand{\br}{\mathrm{Br}}

\newcommand{\brun}{\mathrm{Br}_1}

\newcommand{\al}{\mathrm{al}}
\newcommand{\bral}{\mathrm{Br}_{\al}}

\newcommand{\BM}{\mathrm{BM}}
\newcommand{\gal}{\mathrm{Gal}}
\renewcommand{\sl}{\mathrm{SL}}
\newcommand{\sln}{\mathrm{SL}_n}
\newcommand{\pgln}{\mathrm{PGL}_n}
\newcommand{\gln}{\mathrm{GL}_n}

\newcommand{\pic}{\mathrm{Pic}\,}

\newcommand{\tor}{\mathrm{tor}}
\renewcommand{\ss}{\mathrm{ss}}
\newcommand{\ssu}{\mathrm{ssu}}
\newcommand{\red}{\mathrm{red}}
\newcommand{\torf}{{\rm torf}}
\newcommand{\f}{\mathrm{f}}
\newcommand{\un}{\mathrm{u}}
\newcommand{\qt}{\mathrm{qt}}
\newcommand{\spec}{\mathrm{Spec}\,}
\renewcommand{\cal}[1]{\mathcal{#1}}
\newcommand{\bb}[1]{\mathbb{#1}}
\newcommand{\ov}[1]{\overline{#1}}

\numberwithin{equation}{section}

\usepackage{marvosym}

\title{Local-global principles for homogeneous spaces over some two-dimensional geometric global fields}
\author{Diego Izquierdo\\
\small
\'Ecole polytechnique \\
\small
\texttt{diego.izquierdo@polytechnique.edu}
\and
 Giancarlo Lucchini Arteche\\
 \small
 Universidad de Chile\\
  \small
\texttt{luco@uchile.cl} 
 }
\date{}

\begin{document}

\maketitle

\begin{abstract} In this article, we study the obstructions to the local-global principle for homogeneous spaces with connected or abelian stabilizers over finite extensions of the field $\mathbb{C}((x,y))$ of Laurent series in two variables over the complex numbers and over function fields of curves over $\mathbb{C}((t))$. We give examples that prove that the Brauer-Manin obstruction with respect to the whole Brauer group is not enough to explain the failure of the local-global principle, and we then construct a variant of this obstruction using torsors under quasi-trivial tori which turns out to work. In the end of the article, we compare this new obstruction to the descent obstruction with respect to torsors under tori. For that purpose, we use a result on towers of torsors, that is of independent interest and therefore is proved in a separate appendix.
\end{abstract}

\section{Introduction}

Consider a number field $K$. Recall that a class of $K$-varieties $\mathcal{F}$ is said to satisfy the local-global principle if any variety $Z \in \mathcal{F}$ that has points in every completion of $K$ has in fact a rational point over $K$. The most classical example of such a class is given by the class of quadrics over $K$ (Hasse-Minkowski Theorem).

Some varieties fail to satisfy the local-global principle. In order to explain those failures, given any $K$-variety $Z$, Manin introduced in 1970 a subset $Z(\mathbb{A}_K)^{\mathrm{Br}}$ of the set of adelic points $Z(\mathbb{A}_K)$ that always contains the set of rational points $Z(K)$. This allowed him to define a weakening of the local-global principle, the so-called Brauer-Manin obstruction: the Brauer-Manin obstruction is said to be the only obstruction to the local-global principle for a given class of $K$-varieties $\mathcal{F}$ if any variety $Z \in \mathcal{F}$ for which $Z(\mathbb{A}_K)^{\mathrm{Br}}\neq\emptyset$ has in fact a $K$-rational point.

Since Manin's contribution, various classes have been proved to satisfy or to fail the previous condition. One of the most important examples was given in 1981 by Sansuc (cf.~\cite{Sansuc81}), who proved that the Brauer-Manin obstruction is the only obstruction to the Hasse principle for principal homogeneous spaces under connected linear algebraic groups. This result was then extended in 1996 to all homogeneous spaces under connected linear algebraic groups with connected stabilizers by Borovoi (cf.~\cite{Borovoi96}).\\

Much more recently, there has been a considerable interest in the study of similar questions for various two-dimensional fields naturally arising in geometry. Consider a field $K$ of one of the following two types:
\begin{itemize}
    \item[(a)] the function field of a smooth projective curve $X$ over $\mathbb{C}((t))$;
    \item[(b)] the fraction field of a local, normal, henselian, two-dimensional, excellent domain $A$ with algebraically closed residue field of characteristic $0$ (for instance, $K$ can be any finite extension of the Laurent series field $\mathbb{C}((x,y))$ in two variables over the complex numbers). In this case, $X$ will stand for $\mathrm{Spec}(A)$.
\end{itemize}

The field $K$ then tends to have an arithmetic behaviour similar to that of the usual global fields. In case (a), one can define a local-global principle for the field $K$ by considering its completions with respect to the valuations coming from the set $X^{(1)}$ of closed points of the curve $X$.  Colliot-Th{\'e}l{\`e}ne/Harari (cf.~\cite[\S\S 8.2, 10.2]{ColliotHarari}) then defined a Brauer-Manin obstruction to this local-global principle and proved that it is the only obstruction for principal homogeneous spaces under connected linear algebraic groups.

In case (b), one can define a local-global principle for the field $K$ by considering its completions with respect to the valuations coming from the set $X^{(1)}$ of codimension 1 points of $X$. The first-named author (cf.~\cite[\S 4.1]{IzqDualDim2}) then introduced a Brauer-Manin obstruction to this local-global principle and proved that it is the only obstruction for principal homogeneous spaces under connected linear algebraic groups.

Moreover, in \cite[\S2.3]{ColliotParimalaSuresh}, Colliot-Th{\'e}l{\`e}ne/Parimala/Suresh defined other reciprocity obstructions for fields of types (a) and (b), which turn out to be stronger in case (b) by \cite[Cor.~4.4]{IzqDualDim2}. In other words, every counter-example to the local-global principle explained by the Brauer-Manin obstruction introduced in \cite{IzqDualDim2} can also be explained by the obstructions of \cite{ColliotParimalaSuresh}.\\

The previous results show that principal homogeneous spaces over fields of types (a) and (b) satisfy properties similar to those proved by Sansuc for principal homogeneous spaces over number fields. The present article aims at investigating if these properties extend to the case of homogeneous spaces with connected stabilizers, as it was proved by Borovoi in the case of number fields. In particular, since we intend to use the Brauer-Manin obstructions defined in \cite{ColliotHarari} and \cite{IzqDualDim2}, we will always consider adelic points with respect to the set of places $X^{(1)}$.

As we will explain at the beginning of Section 4, one can easily find examples of homogeneous spaces $Z$ under $\sln$ with toric stabilizers that have no rational points but for which the Brauer-Manin set $Z(\mathbb{A}_K)^{\mathrm{Br}}$ is non-empty (see Section \ref{sec preliminaires} for precise definitions). This follows from the failure of the local-global principle with respect to $X^{(1)}$ for central simple algebras over fields of types (a) and (b) (cf.~\cite[\S2.3]{ColliotHarari} and \cite[\S2]{IzqDualDim2}). So, in order to understand the obstructions to the local-global principle for such homogeneous spaces, one needs to impose extra local conditions on $X$.

This is easy to do. Indeed, the Brauer-Manin obstruction we are using here only takes into account some of the completions of the field $K$: in case (a), we are only considering completions with respect to the valuations coming from the closed points of the curve $X$, while in case (b), we are only considering completions with respect to the valuations coming from prime ideals of height one in the domain $A$. However, the field $K$ has more valuations, and hence it is natural to ask whether a homogeneous space $Z/K$ with connected stabilizers for which the Brauer-Manin set $Z(\mathbb{A}_K)^{\mathrm{Br}}$ is non-empty \emph{and} that has points in \emph{all} completions of $K$ always has a rational point. The first main theorem of the present article gives a negative answer to this question:

\begin{thm}[Consequence of Theorem \ref{thm counter} and Remark \ref{rem exist counter}] In each of the cases (a) and (b), there exists a field $K$ and a homogeneous space $Z/K$ under $\sl_{n,K}$ for some $n \geq 1$ with toric stabilizers for which the Brauer-Manin set $Z(\mathbb{A}_K)^{\mathrm{Br}}$ is non-empty, that has points in all completions of $K$, but that has no $K$-rational points.
\end{thm}

For that reason, one should go beyond the Brauer-Manin obstruction in order to understand the failure of the local-global principle for homogeneous spaces over fields of types (a) and (b). A natural way to do so consists in combining the Brauer-Manin obstruction with another very usual one, the descent obstruction. If it is known that, for smooth and geometrically integral varieties over number fields, the descent obstruction with respect to torsors under connected linear groups does not carry more information than the Brauer-Manin obstruction with respect to the whole Brauer group (cf.~\cite{Har02}), the situation turns out to be completely different over fields of types (a) and (b):

\begin{thm}[Consequence of Theorem \ref{thm new obst} and Proposition \ref{prop conn stab}] \label{mt2} Let $K$ be a field of type (a) or (b) as above. Let $Z$ be a homogeneous space under a connected linear group $G$ with connected geometric stabilizers. Assume that $Z$ has points in every completion of $K$ with respect to a discrete valuation. Then there exists a torsor $W\rightarrow Z$ under a quasi-trivial torus $T$ such that the following statements are equivalent:
\begin{itemize}
    \item[(i)] the set $Z(K)$ is non-empty;
    \item[(ii)] the set $W(\bb A_K)^{\br(W)}$ is non-empty.
\end{itemize}
\end{thm}

The torsor $W$ is obtained by using some results in \cite{DLA}, but it can be explicitly given using Galois descent. Moreover, while proving Theorem \ref{mt2}, we will see that the assumption that $Z$ has points in every completion of $K$ with respect to a discrete valuation can be replaced by the assumption that $Z$ has points in an explicit finite number of completions of $K$. We will also see that one needs to take into account not the whole Brauer group of $W$ in the condition $W(\bb A_K)^{\br(W)}\neq\emptyset$, but only a subquotient that turns out to be finite. Hence, even if we do not know whether all the involved constructions are algorithmically computable, Theorem \ref{mt2} provides finitely many explicit conditions to find out whether the homogeneous space $Z$ has rational points.

At the end of the article, we compare the obstruction of Theorem \ref{mt2}, which combines the Brauer-Manin obstruction and the descent obstruction with respect to torsors under quasi-trivial tori, to the descent obstruction with respect to torsors under general tori (see theorem \ref{thm tor vs qtBr}). We deduce:

\begin{thm}[Consequence of Corollary \ref{descent hom} and Proposition \ref{prop conn stab}]
Let $K$ be a field of type (a) or (b) as above. Let $Z$ be a homogeneous space under a connected linear group $G$ with connected geometric stabilizers. Assume that $Z$ has points in every completion of $K$ with respect to a discrete valuation and that the subset of $Z(\bb A_K)$ given by the descent obstruction with respect to torsors under general tori (as defined in section \ref{descent sec}) is non-empty. Then $Z$ has a $K$-rational point.
\end{thm}

The proof of this result relies on a theorem of independent interest about towers of torsors stating the following. If $K$ is a field of characteristic $0$, $G$ is a connected linear $K$-group with trivial geometric Picard group, $T$ is a $K$-torus and $X$ is a geometrically integral $K$-variety, then every $T$-torsor over a $G$-torsor over $X$ admits a structure of a torsor over $X$ under a certain extension of $G$ by $T$. This result will be proved in Appendix \ref{sec appendix}.

\section*{Acknowledgements}
The authors would like to thank Mathieu Florence for his enormous help with the result in the appendix. They would also like to thank Jean-Louis Colliot-Thélène and an anonymous referee for all their suggestions that allowed them to considerably improve this text. The second author would like to thank Michel Brion for helpful discussions. The second author's research was partially supported by ANID via FONDECYT Grant 1210010 and PAI Grant 79170034.

\section{Notations and preliminaries}\label{sec preliminaires}

In this section we fix the notations that will be used throughout this article.

\paragraph*{Fields}
Let $F$ be an algebraically closed field of characteristic $0$, and let $A$ be a local, integral, normal, henselian, excellent domain with residue field $F$ and fraction field $E$. Set $S:=\mathrm{Spec}(A)$ and let $s$ be the closed point of $S$. Consider an integral, regular, $2$-dimensional scheme $\mathcal{X}$ endowed with a projective surjective morphism $p: \mathcal{X} \rightarrow S$. Let $K$ be the function field of $\mathcal{X}$, let $X_0$ be the special fiber of $p$ and set $X:=\mathcal{X} \setminus X_0 = p^{-1}(S \setminus \{s\})$. \\

\noindent\emph{Assumptions.} In the sequel, we will assume that $X_0$ is a strict normal crossings divisor in $\mathcal{X}$ and that we are in one of the two following cases:
\begin{itemize}
\item[(a)] \emph{(Semi-global case).} The ring $A$ is a discrete valuation ring, all fibers of $p$ are 1-dimensional, and the generic fiber is smooth and geometrically integral. 
\item[(b)] \emph{(Local case).} The ring $A$ is $2$-dimensional and $p$ is birational. In particular, $X=S\setminus\{s\}$.
\end{itemize}

Throughout the article, we will say that the field $K$ is a field of type (a) or (b) respectively. Observe that these types of fields are stable under finite extensions. Indeed, if $L/K$ is a finite extension, then $L$ is a semi-global field in the sense of (a) (resp. a local field in the sense of (b)) if, and only if, $K$ is a semi-global field in the sense of (a) (resp. a local field in the sense of (b)).

\begin{exa}
\begin{itemize}
\item[(i)] In the semi-global case, one can take $A=\mathbb{C}[[t]]$, let $X$ be a smooth projective geometrically integral curve over $E=\mathbb{C}((t))$ and $\mathcal{X}$ be a regular model of $X$ whose special fiber has strict normal crossings. The field $K$ is then a finite extension of $E(x)$.
\item[(ii)] In the local case, one can let $K$ be a finite extension of $\mathbb{C}((x,y))$, let $S$ be the normalization of $\mathrm{Spec}(\mathbb{C}[[x,y]])$ in $K$ and let $\mathcal{X}$ be a desingularization of $X$ whose special fiber has strict normal crossings.
\end{itemize}
\end{exa}

\paragraph*{Local-global principle and Brauer-Manin obstruction}
A place of $K$ is a discrete valuation of rank 1 of $K$. Given a subset $\Omega$ of the set $\Omega_K$ of \emph{all} places of $K$, we say that a $K$-variety $Z$ satisfies the local-global principle with respect to $\Omega$ if
\[\prod_{v\in \Omega} Z(K_v) \neq \emptyset \,\Rightarrow\, Z(K) \neq \emptyset,\]
where $K_v$ denotes the completion of $K$ at the place $v$.

There are mainly three different natural choices for the set $\Omega$ in this context:
\begin{itemize}
    \item[(a)] the set $\Omega_K$ of all places of $K$;
    \item[(b)] the set $\mathcal{X}^{(1)}$ of codimension 1 points of $\mathcal{X}$;
    \item[(c)] the set $X^{(1)}$ of codimension 1 points of $X$.
\end{itemize}
Whatever the choice, one can find varieties that fail the local-global principle. That is why one usually tries to introduce obstructions that explain such failures. Fix a $K$-variety $Z$. Given a regular model $\mathcal{Z}$ of $Z$ over an open subset $U$ of $X$, define the set of its adelic points as
\[Z(\mathbb{A}_K)= \left\lbrace(P_v) \in \prod_{v\in X^{(1)}} Z(K_v) \, : \, P_v \in \mathcal{Z}(\mathcal{O}_v) \,\text{for almost all $v$}\right\rbrace.\]
Recall that this set does not depend on the choice of the model $\mathcal{Z}$ of $Z$. By \cite[Prop.~2.1.(v)]{ColliotHarari} (semi-global case) and \cite[Thm.~1.6]{IzqDualDim2} (local case), there is an exact sequence
\[\br(K) \rightarrow \bigoplus_{v\in X^{(1)}} \br(K_v) \xrightarrow{\theta} \mathbb{Q}/\mathbb{Z}\to 0.\]
Now, for $v\in X^{(1)}$, the residue field of the ring of integers $\cal O_v$ of $K_v$ has cohomological dimension 1, thus $\br(\cal O_v)=0$ for every $v\in\Omega$. Hence one can define a Brauer-Manin pairing
\begin{align*}
\mathrm{BM}: Z(\mathbb{A}_K) \times \br(Z) & \rightarrow \mathbb{Q}/\mathbb{Z}\\
\left( (P_v),\alpha \right) & \mapsto \theta(P_v^*\alpha)
\end{align*}
with the following property: $Z(K)$ is contained in the subset $Z(\mathbb{A}_K)^{\br(Z)}$ of $Z(\mathbb{A}_K)$ defined as the set of adelic points that are orthogonal to the cohomological Brauer group $\br(Z):=H^2(Z,\mathbb{G}_m)$. Define
\begin{align*}
    \br_0(Z) &:=\im \left( \br(K) \rightarrow \br(Z) \right),\\
    \br_1(Z) &:=\ker \left( \br(Z) \rightarrow \br(Z_{\bar K}) \right),\\
    \bral(Z) &:=\br_1(Z)/\br_0(Z),
\end{align*}
and note that the pairing $\mathrm{BM}$ factors through $\br(Z)/\br_0(Z)$. Given a subgroup $B$ of $\br(Z)$ or of $\br(Z)/\br_0(Z)$, we say that the Brauer-Manin obstruction with respect to $B$ is the only obstruction for $Z$ if
\[Z(\mathbb{A}_K)^{B} \neq \emptyset\, \Rightarrow\, Z(K) \neq \emptyset.\]
In the sequel, the group $B$ will often be chosen to be:
$$\Brusse (Z) := \ker \left( \bral(Z) \rightarrow \prod_{v\in X^{(1)}} \bral(Z_{K_v}) \right).$$
This group has already been used in the articles \cite{ColliotHarari}, \cite{IzqSMF} and \cite{IzqDualDim2}, where it was proved that the Brauer-Manin obstruction with respect to $\Brusse(Z)$ is the only obstruction to the local-global principle for principal homogeneous spaces under connected linear groups over $K$.

\paragraph*{Tate-Shafarevich groups}
As explained in the previous paragraph, the local-global principle over the field $K$ can be defined with respect to different sets of places. For that reason, in the sequel, when we are given a Galois module $M$ over $K$ and an integer $i \geq 0$, we will make use of the following three Tate-Shafarevich groups:
\begin{gather*}
    \Sha^i(K,M):= \ker \left( H^i(K,M) \rightarrow \prod_{v\in X^{(1)}} H^i(K_v,M) \right),\\
    \Sha^i_{\mathcal{X}}(K,M):= \ker \left( H^i(K,M) \rightarrow \prod_{v\in \mathcal{X}^{(1)}} H^i(K_v,M) \right),\\
    \Sha^i_{\mathrm{tot}}(K,M):= \ker \left( H^i(K,M) \rightarrow \prod_{v\in \Omega_K} H^i(K_v,M) \right).
    \end{gather*}

\paragraph*{Algebraic groups and homogeneous spaces}
For a linear algebraic $K$-group $G$, we use the following notations:
\begin{itemize}
\item $D(G)$ is the derived subgroup of $G$;
\item $G^\circ$ is the neutral connected compontent of $G$;
\item $G^\f=G/G^\circ$ is the group of connected components of $G$ (it is a finite group);
\item $G^\un$ is the unipotent radical of $G^\circ$ (it is a unipotent characteristic subgroup);
\item $G^\red=G^\circ/G^\un$ (it is a reductive group);
\item $G^\ss=D(G^\red)$ (it is a semisimple group);
\item $G^\tor=G^\red/G^\ss$ (it is a torus);
\item $G^\ssu=\ker[G^\circ\twoheadrightarrow G^\tor]$ (it is an extension of $G^\ss$ by $G^\un$);
\item $G^\torf=G/G^\ssu$ (it is an extension of $G^\f$ by $G^\tor$);
\item $\hat G$ the Galois module of the geometric characters of $G$.
\end{itemize}
A torus $T$ is said to be quasi-trivial if $\hat T$ is an induced $\gal(\bar K/K)$-module.\\

In the sequel, we will be often considering homogeneous spaces of a connected linear group $G$ with geometric stabilizer $\bar H$ satisfying the following hypothesis:
\begin{equation}\label{eq hyp}
G^\ss\text{ is simply connected and }\bar H^{\torf}\text{ is abelian.}
\end{equation}
This is the same hypothesis used by Borovoi in \cite{Borovoi96}. In particular, since for connected $\bar H$ we have that $\bar H^\torf=\bar H^\tor$ is abelian, the following result follows immediately from \cite[Lem.~5.2]{Borovoi96}.

\begin{pro}\label{prop conn stab}
Let $Z$ be a homogeneous space of a connected linear group with connected geometric stabilizers. Then $Z$ is a homogeneous space of a connected linear $K$-group $G$ with geometric stabilizer $\bar H$ satisfying \eqref{eq hyp}.
\end{pro}

Let $Z$ be a homogeneous space under a linear connected $K$-group $G$ with geometric stabilizer $\bar H$. Following \cite[\S2]{DLA}, we associate to $Z$ a gerbe $\cal M_Z$ and an injective morphism of gerbes $\cal M_Z\to\underline{\mathrm{TORS}}(G)$, where $\underline{\mathrm{TORS}}(G)$ denotes the trivial gerbe of $G$-torsors under $G$. The gerbe $\cal M_Z$ is known as the \emph{Springer class} of $Z$ and can be regarded as a class in a certain non-abelian 2-cohomology set associated to the geometric stabilizer $\bar H$. It corresponds to the obstruction for $Z$ to being dominated by a principal homogeneous space of $G$ (i.e.~the class is neutral if and only if such a dominating torsor exists). See \cite{DLA} for more details.

Assume now that $G$ and $\bar H$ satisfy \eqref{eq hyp}. Since the subgroup $\bar H^\ssu$ is characteristic in $\bar H$, we may also consider the induced gerbe $\cal M^\torf_Z$ which, since $\bar H^\torf$ is abelian, defines a natural $K$-form $H^\torf$ of $\bar H^\torf$ (cf.~for instance \cite[\S1.7]{Borovoi93}) and $\cal M^\torf_Z$ can then be regarded as a class $\eta^\torf\in H^2(K,H^\torf)$.\\

We will use in the sequel the following results on these particular base fields and their corresponding nonabelian cohomology sets for semisimple groups.

\begin{pro}\label{prop coh dim 2 and period-index}
Let $K$ be a field of type (a) or (b). Then $K$ has cohomological dimension 2 and for each central simple algebra $A$ over $K$, its index and its period are the same. Moreover, every principal homogeneous space under a semisimple simply connected $K$-group has a rational point.
\end{pro}

\begin{proof}
For fields of type (a), the first two assertions are a direct consequence of \cite[Thm.~5.5]{HHK} and the last one then follows from \cite[Thm.~1.2]{ColliotGilleParimala}. For fields of type (b), all assertions are contained in \cite[Thm.~1.4]{ColliotGilleParimala}.
\end{proof}

\begin{lem}\label{lem GA}
Let $K$ be a field of type (a) or (b). Let $\cal L$ be a $K$-lien (or $K$-kernel) whose underlying $\bar K$-group $\bar G$ satisfies $\bar G=\bar G^\ssu$. Then every class in $H^2(K,\cal L)$ is neutral.
\end{lem}

\begin{proof}
By \cite[Prop.~4.1]{Borovoi93}, we may assume that $\bar G$ is reductive, hence semisimple. Then \cite[Prop.~3.1]{Borovoi93} tells us that there exists a neutral class in $H^2(K,\cal L)$, so that we can regard $H^2(K,\cal L)$ as $H^2(K,G)$ for some $K$-form $G$ of $\bar G$. The lemma is then a direct consequence of results by Gonz\'alez-Avil\'es. Indeed, by Proposition \ref{prop coh dim 2 and period-index} and \cite[Ex.~5.4(vi)]{GonzalezModulesCroises}, the field $K$ is ``of Douai type''. This allows us to apply \cite[Thm.~5.8(ii)]{GonzalezModulesCroises}, which tells us that a class in $H^2(K,\cal L)$ is neutral if and only if its canonical image in $H^2(K,G^\tor)$ is trivial. But $G^\tor=1$ and the result follows.
\end{proof}

\section{Local-global principle with Brauer-Manin obstruction}

In this section we study a class of homogeneous spaces for which we can prove that the Brauer-Manin obstruction is the only obstruction to the local-global principle. The strategy is to use the validity of this assertion for principal homogeneous spaces and slice our homogeneous space in order to reduce ourselves to this case.\\

We start with a case where there is no need for any obstruction (and not even a local-global principle!). Recall that fields of type (a) and (b) were defined at the beginning of Section \ref{sec preliminaires}.

\begin{pro}\label{prop eh ssu}
Let $K$ be a field of type (a) or (b). Let $Z$ be a homogeneous space under a linear group $G$ with geometric stabilizer $\bar H$ satisfying \eqref{eq hyp}. Assume moreover that $G=G^\ssu$ and $\bar H=\bar H^\ssu$. Then $Z$ has a rational point.
\end{pro}

\begin{proof}
We prove first that we may assume $G=G^\ss$. Indeed, using \cite[Lem.~3.1]{Borovoi96}, we may consider the map $Z\to Z/G^\un$, where $Z/G^\un$ is a homogeneous space of $G^\ss$ and the fibers are homogeneous spaces of $G^\un$. Now, the Springer class of such a fiber lies in $H^2(K,U)$ for a certain \emph{unipotent} $K$-lien (or $K$-kernel) $U$, which only contains neutral classes by \cite[Cor.~4.2]{Borovoi93}. This implies that the fiber is dominated by a torsor under $G^\un$, which clearly has rational points since $H^1$ is trivial for unipotent groups. Thus, if $Z/G^\un$ has rational points, $Z$ has rational points as well.

Assume then that $G=G^\ss$ and let us start with the case of trivial stabilizers. By Proposition \ref{prop coh dim 2 and period-index}, we know that principal homogeneous spaces of $G=G^\ss$ always have rational points when $G^\ss$ is simply connected. The result for non-trivial stabilizers follows from Lemma \ref{lem GA}, since it implies that every such homogeneous space is dominated by a $G$-torsor.
\end{proof}

We allow now some tori to appear, with a technical hypothesis that ensures that the Brauer-Manin obstruction controls the local-global principle.

\begin{thm}\label{thm Htorf inj Gtor}
Let $K$ be a field of type (a) or (b). Let $Z$ be a homogeneous space under a connected linear group $G$ with geometric stabilizer $\bar H$ satisfying \eqref{eq hyp}. Assume that the natural arrow $\bar H^\torf\to\bar G^\tor$ is injective. Then
\[Z(\bb A_K)^{\Brusse(Z)} \neq\emptyset\,\Rightarrow\, Z(K)\neq\emptyset.\]
\end{thm}

\begin{proof}
Since $G$ acts on $Z$, we may consider the quotient variety $Z'=Z/G^\ssu$. Since $\bar H^\torf\to\bar G^\tor$ is injective, we know by \cite[Lem.~3.1]{Borovoi96} that $Z'$ is a homogeneous space of $G^\tor$ with geometric stabilizer $\bar H^\torf$ and the geometric fibers of the quotient morphism $Z\to Z'$ are homogeneous spaces of $\bar G^\ssu$ with geometric stabilizers (isomorphic to) $\bar H^\ssu$. Consider then an adelic point $(P_v)\in Z(\bb A_K)^{\Brusse (Z)}$ and push it to $Z'$. By the functoriality of the Brauer pairing, its image lies in $Z'(\bb A_K)^{\Brusse(Z')}$. Since $Z'$ is a homogeneous space of a torus, it is also a principal homogeneous space of another torus (namely, of  $G^\tor/H^\torf$, where $H^\torf$ is the $K$-form of $\bar H^\torf$ associated to $Z$) and hence, by \cite[Cor.~8.3]{ColliotHarari} (semi-global case) and \cite[Thm.~4.2]{IzqDualDim2} (local case), we know that there is a $K$-point $P_0\in Z'(K)$. The fiber above $P_0$ is a homogeneous space satisfying the hypotheses of Proposition \ref{prop eh ssu}, hence it has a rational point, which implies that $Z(K)\neq\emptyset$.
\end{proof}

As we will see in the subsequent sections, this is as far as we can go with the Brauer-Manin obstruction, in the sense that under more general hypotheses the Brauer-Manin obstruction will not be enough to explain counterexamples to the local-global principle for homogeneous spaces.

We conclude this section with a result that tells us that the Brauer-Manin obstruction with respect to $\Brusse(Z)$ factors through a finite quotient.

\begin{pro}\label{prop fact brusse barrre}
Let $K$ be a field of type (a) or (b). Let $Z$ be a homogeneous space under a connected linear group $G$ with geometric stabilizer $\bar H$ and satisfying \eqref{eq hyp}. Assume that the natural arrow $\bar H^\torf\to\bar G^\tor$ is injective. Then the Brauer pairing
\[Z(\bb A_K)\times \Brusse(Z)\to \Q/\Z,\]
factors on the right through the quotient $\ov{\Brusse(Z)}$ of $\Brusse(Z)$ by its maximal divisible subgroup.
\end{pro}

\begin{proof}
Consider, as in the last proof, the quotient variety $Z'=Z/G^\ssu$, which is a principal homogeneous space under the torus $G^\tor/H^\torf$. By \cite[Thm.~2.8]{IzqSMF} (semi-global case) and \cite[Thm.~4.2]{IzqDualDim2} (local case), we know that the result holds for $Z'$. Since $G^\tor=(G/G^\ssu)^\tor$, we see by \cite[Thm.~7.2]{BorovoiVH} that
\[\bral(Z)\cong H^2(K,[\hat G^\tor\to \hat H^\torf])\cong \bral(Z').\]
We get then an isomorphism $\Brusse(Z')\xrightarrow{\sim}\Brusse(Z)$. The functoriality of the Brauer pairing gives us then the result for $Z$.
\end{proof}

The proof of this result tells us that in fact $\overline{\Brusse(Z)}$ is isomorphic to $\overline{\Brusse(Z')}$ with $Z'$ a torsor under a $K$-torus, and we know by Poitou-Tate duality (cf.~\cite[Thm.~7.2]{ColliotHarari} and \cite[Thm.~4.8]{IzqDualDim2}) that this group is finite.

\section{Counter-examples}

As stated in the Introduction, one can easily produce examples of homogeneous spaces $Z$ over $K$ that fail the local-global principle with respect to the small set of places $X^{(1)}$ and such that $Z(\bb A_K)^{\br(Z)}$ is nonempty. Indeed, consider the quasi-trivial $K$-torus $T=R_{L/K}(\gm)$ for some finite extension $L/K$ such that $\Sha^2(L,\gm)\neq 0$ (this can be done by \cite[\S~2.3]{ColliotHarari} and \cite[\S~1]{IzqDualDim2}). Take a non-zero class $\alpha \in \Sha^2(K,T)=\Sha^2(L,\gm)$. By \cite[Prop.~2.1 \& Cor.~3.3]{DLA}, there exists a positive integer $n$, a homogeneous space $Z/L$ under $\sl_{n,K}$ and a geometric point $z \in Z(\bar{K})$ such that the torus $T_{\bar{K}}$ is the stabilizer of $z$ and $\alpha$ is the Springer class of $Z$. Since $\alpha \neq 0\in \Sha^2(K,T)$, the set $Z(K)$ is empty and the set $Z(K_v)$ is non-empty for every discrete valuation $v\in X^{(1)}$. Moreover, by \cite[Thm.~7.2]{BorovoiVH}, we have $\bral(Z)\simeq H^1(K,\hat T)=H^1(L,\Z)=0$. Finally, since $\bar{Z}\cong \sl_{n,\bar{K}}/\bar{T}$, the generically trivial $\bar{T}$-torsor $\sl_{n,\bar{K}}\rightarrow \bar{Z}$ induces an injection of Brauer groups $\mathrm{Br}(\bar{Z}) \hookrightarrow \mathrm{Br}(\sl_{n,\bar{K}})$. But $\mathrm{Br}(\sl_{n,\bar{K}})=0$ by \cite[\S0]{SGilleBrGss}, and hence $\mathrm{Br}(\bar{Z})=0$. We have then trivially $Z(\bb A_K)^{\br(Z)}=Z(\bb A_K)\neq \emptyset$.\\

These examples suggest that, in order to get an obstruction that controls the local-global principle for homogeneous spaces with connected stabilizers, one needs to take into account local information coming from places outside $X^{(1)}$. However, as we prove here below, after replacing $K$ by a finite extension, one can construct homogeneous spaces $Z$ that fail the local-global principle with respect to the set $\Omega_K$ of \emph{all} places of $K$ and such that $Z(\bb A_K)^{\br(Z)}$ is nonempty. In further sections, we will develop more involved obstructions using torsors under quasi-trivial tori, that explain all these failures of the local-global principle.\\

The main result of this section is the following:

\begin{thm}\label{thm counter}
Let $K$ be a field of type (a) or (b). Assume that the special fiber $X_0$ contains three smooth and integral divisors $L_1,L_2,L_3$ such that, for each $i\neq j$, the divisors $L_i$ and $L_j$ intersect at a single point and the intersection is transversal. Then there exists a homogeneous space $Z/K$ under $\sl_{n,K}$ and a geometric point $z \in Z(\bar{K})$ satisfying the following conditions:
\begin{itemize}
\item[(i)] the (geometric) stabilizer of $z$ is a torus;
\item[(ii)] for any discrete valuation $v$ on $K$, the set $Z(K_v)$ is non-empty;
\item[(iii)] the set $Z(\mathbb{A}_K)^{\br(Z)}$ is non-empty;
\item[(iv)] the set $Z(K)$ is empty.
\end{itemize}
\end{thm}

\begin{rem}\label{rem exist counter}
For any field $K$ of type (a) or (b) one can construct a finite extension satisfying the assumption in Theorem \ref{thm counter}. The interested reader can find explicit examples of such fields by looking at \cite[Ex.~5.6. \& 5.8]{ColliotParimalaSuresh}.
\end{rem}

The proof, which has 8 steps, goes as follows: In Step 1 we fix some notations associated to the field $K$ and the divisors $L_i$. In Steps 2--4, we define a torus $S$ and a homogeneous space $Z$ whose Springer class lies in $\Sha^2_\mathrm{tot} (K,S)$, which ensures that properties (i), (ii) and (iv) hold. In Step 5, we construct auxiliary varieties
\[\xymatrix@C=.5em{
& W \ar[rd] \ar[dl] & & \tilde V \ar[dl] \\
Z & & V & 
}\]
where all arrows are torsors under $K$-rational groups. This gives us isomorphisms of the corresponding unramified Brauer groups. Moreover, both $V$ and $\tilde V$ are torsors under suitable $K$-tori that come from constructions by Colliot-Th\'el\`ene in \cite{ColliotBordeaux}, where their unramified Brauer groups were explicitly given. This allows us to compute the unramified Brauer group of $Z$ in Step 6. Step 7 uses this information in order to compute explicitly the (whole) Brauer group of $Z$. Finally, we use this in order to check (iii) in Step 8.

\begin{proof}

\emph{Step 1: Recalling the constructions from \cite[\S5.2]{ColliotParimalaSuresh}.} Set:
$$m_1:= L_2 \cap L_3,\;\;\;\;\;m_2:= L_1 \cap L_3,\;\;\;\;\;m_3:= L_1 \cap L_2.$$
 Since the Picard group of a semi-local ring is always trivial  one can construct three elements $\pi_1,\pi_2,\pi_3\in K^{\times}$ in the following way: 
\begin{itemize} 
\item[(a)] First, one chooses $\pi_1 \in K^{\times}$ so that the support of $D_1:=\mathrm{div}_{\mathcal{X}}(\pi_1)-L_1$ does not contain any of the $m_i$'s. 
\item[(b)] Secondly, one chooses a place $v_0 \in \mathcal{X}^{(1)} \setminus X_0^{(0)}$ that is not contained in the support of $D_1$. One can then choose $\pi_2 \in K^{\times}$ so that:
\begin{itemize}
\item[(i)] $v_0(\pi_2)=1$,
\item[(ii)] the support of $D_2:=\mathrm{div}_{\mathcal{X}}(\pi_2)-L_2$ does not contain any of the $m_i$'s, any of the components of the support of $D_1$ and any of the intersection points of one of the $L_i$'s with the support of $D_1$. 
\end{itemize}
\item[(c)] Finally, one chooses $\pi_3 \in K^{\times}$ so that:
\begin{itemize}
\item[(i)] the classes of $\pi_1$ and $\pi_3$ in the quotient $k(v_0)^{\times}/{k(v_0)^{\times}}^2$ are not equal,
\item[(ii)] the support of $D_3:=\mathrm{div}_{\mathcal{X}}(\pi_3)-L_3$ does not contain any of the $m_i$'s, any of the components of the support of $D_1$, any of the components of the support of $D_2$, any of the intersection points of one of the $L_i$'s with the support of $D_1$, any of the intersection points of one of the $L_i$'s with the support of $D_2$ and any of the intersection points of the supports of $D_1$ and $D_2$. 
\end{itemize}
\end{itemize}

\emph{Step 2: Defining the torus $S$ as well as some other useful auxiliary tori.} Set $a:=\pi_2\pi_3$ and $b:=\pi_3\pi_1$, introduce the field $M:=K(\sqrt{a},\sqrt{b})$, and consider the normic torus $Q$ and the multinormic torus $\tilde{Q}$ defined by the following exact exact sequences:
\begin{gather}
    1 \rightarrow Q \rightarrow R_{M/K}(\mathbb{G}_m) \rightarrow \mathbb{G}_m \rightarrow 1,\label{eq Colliot 1} \\ 
    1 \rightarrow \tilde{Q} \rightarrow  R_{K(\sqrt{a})/K}(\mathbb{G}_m)\times R_{K(\sqrt{b})/K}(\mathbb{G}_m) \times R_{K(\sqrt{ab})/K}(\mathbb{G}_m) \rightarrow \mathbb{G}_m \rightarrow 1.
    \end{gather}
Note that, by \cite[Prop.~3.1(b)]{ColliotBordeaux}, we have an exact sequence:
\begin{equation}\label{eq Colliot}
    1 \rightarrow \mathbb{G}_m^2 \rightarrow \tilde{Q} \rightarrow Q \rightarrow 1.
    \end{equation}
We can now introduce a torus $S$ by writing down a resolution of $Q$
\begin{equation}\label{flasque}
1 \rightarrow S \rightarrow R \rightarrow Q \rightarrow 1,
\end{equation}
with $R$ is a quasi-trivial $K$-torus split by $M$. In particular, $S$ is also split by $M$. \\

\emph{Step 3: Computing the Tate-Shafarevich group $\Sha^2_{\mathrm{tot}}(K,S)$.} By Shapiro's Lemma and Hilbert's Theorem 90, we have $H^1(K,R)=0$. Hence the short exact sequence (\ref{flasque}) induces an exact sequence:
$$0 \rightarrow \Sha^1_{\mathrm{tot}}(K,Q) \rightarrow \Sha^2_{\mathrm{tot}}(K,S) \rightarrow \Sha^2_{\mathrm{tot}}(K,R).$$
Moreover, by \cite[Cor.~5.3 \& 5.7]{ColliotParimalaSuresh}, which apply to the torus $Q$ from Step 2, the group $\Sha^1_{\mathrm{tot}}(K,Q)$ is non trivial. Hence $\Sha^2_{\mathrm{tot}}(K,S)$ is non trivial.\\

\emph{Step 4: Defining the homogeneous space $Z$ and proving statements (i), (ii) and (iv).} Take a non-zero class $\alpha \in \Sha^2_{\mathrm{tot}}(K,S)$. As before, by \cite[Prop.~2.1 \& Cor.~3.3]{DLA}, there exists a positive integer $n$, a homogeneous space $Z/K$ under $\sl_{n,K}$ and a geometric point $z \in Z(\bar{K})$ such that the torus $S_{\bar{K}}$ is the stabilizer of $z$ and $\alpha$ is the Springer class of $Z$. Since $\alpha \neq 0$, the set $Z(K)$ is empty. Since $\alpha \in \Sha^2_{\mathrm{tot}}(K,S)$, the set $Z(K_v)$ is non-empty for every discrete valuation $v$ on $K$. We have therefore settled (i), (ii) and (iv). The remaining steps will be devoted to the proof of statement (iii). \\

\emph{Step 5: Constructing an auxiliary torsor $\tilde{V}$ under $\tilde{Q}$ that is stably birational to $Z$.} By Proposition \ref{prop geom const}, which will be proved in next section, one can construct a $K$-homogeneous space $W$ under $\sl_{n,K}\times R$  together with a morphism of $K$-homogeneous spaces $p:W\rightarrow Z$ which makes $W$ into an $R$-torsor. Since $R$ is quasi-trivial and $\prod_{v\in \Omega_K} Z(K_v)\neq \emptyset$, Hilbert's theorem 90 shows that $\prod_{v\in \Omega_K} W(K_v)\neq \emptyset$. 

Now consider the quotient $V:=W/\sl_{n,K}$ and the natural projection $q:W \rightarrow V$. The $K$-variety $V$ is a $Q$-torsor such that $\prod_{v\in \Omega_K} V(K_v)\neq \emptyset$. Its class in $H^1(K,Q)$ therefore belongs to $\Sha^1_{\mathrm{tot}}(K,Q)$. Moreover, the short exact sequence \eqref{eq Colliot} and the triviality of the group $H^1(K,\mathbb{G}_m)$ induce an exact sequence of Tate-Shafarevich groups:
$$0 \rightarrow \Sha^1_{\mathrm{tot}}(K,\tilde{Q}) \rightarrow \Sha^1_{\mathrm{tot}}(K,Q) \rightarrow \Sha^2_{\mathrm{tot}}(K,\mathbb{G}_m)^2=0,$$
and hence an isomorphism $\Sha^1_{\mathrm{tot}}(K,\tilde{Q}) \cong \Sha^1_{\mathrm{tot}}(K,Q)$. The $Q$-torsor $V$ therefore comes from a $\tilde{Q}$-torsor $\tilde{V}$ that has points in every completion of $K$. The natural morphism $r:\tilde{V}\rightarrow V$ is a $\mathbb{G}_m^2$-torsor. \\

\emph{Step 6: Using the torsor $\tilde{V}$ to compute the unramified Brauer group of $K(Z)/K$.} Since the groups $R$, $\sl_{n,K}$ and $\mathbb{G}_m^2$ are all $K$-rational, we have the following isomorphisms of unramified Brauer groups:
$$\mathrm{Br}_{\mathrm{nr}}(K(Z)/K) \xrightarrow[p^*]{\cong} \mathrm{Br}_{\mathrm{nr}}(K(W)/K) \xleftarrow[q^*]{\cong} \mathrm{Br}_{\mathrm{nr}}(K(V)/K) \xrightarrow[r^*]{\cong} \mathrm{Br}_{\mathrm{nr}}(K(\tilde{V})/K). $$
But the $\tilde{Q}$-torsor $\tilde{V}$ is given by an equation of the form:
$$(x_1^2-ax_2^2)(y_1^2-by_2^2)(x_3-aby_3^2)=c,$$
for some $c \in K^{\times}$, and then, by \cite[Thm.~4.1]{ColliotBordeaux}, the quotient $\mathrm{Br}_{\mathrm{nr}}(K(\tilde{V})/K)/\im(\mathrm{Br}(K))$ is a cyclic group of order $2$ generated by the class of the quaternion algebra $A:=(x_1^2-ax_2^2,b)$. Hence $\mathrm{Br}_{\mathrm{nr}}(K(Z)/K)/\im(\mathrm{Br}(K))$ is also a cyclic group of order $2$ generated by $\gamma:=(p^*)^{-1}q^*(r^*)^{-1}([A])$.\\

\emph{Step 7: Computing the Brauer group of $Z$.}
By Step 6, the algebraic Brauer group $\bral(Z)$ contains the non-zero class $\gamma$ (note that $A$ is algebraic). Let us prove that $\bral(Z) = \{0,\gamma\}$. Since $\bral(Z)\cong H^1(K,\hat{S})$ by \cite[Thm.~7.2]{BorovoiVH}, it suffices to check that $H^1(K,\hat{S})$ has order at most $2$. By inflation-restriction, we have an isomorphism:
$$ H^1(M/K,\hat{S}) \cong  H^1(K,\hat{S}) .$$
Moreover, by definition of the torus $S$, we have an exact sequence of $\gal(M/K)$-modules:
$$0 \rightarrow \hat{Q} \rightarrow \hat{R} \rightarrow \hat{S} \rightarrow 0$$
that induces an exact sequence:
$$0 = H^1(M/K,\hat{R}) \rightarrow H^1(M/K,\hat{S}) \rightarrow H^2(M/K,\hat{Q}) \rightarrow H^2(M/K,\hat{R}).$$
Moreover, by dualizing the exact sequence \eqref{eq Colliot 1}, we get:
$$H^2(M/K,\hat{Q}) \cong H^3(M/K,\mathbb{Z}) \cong H^3\left( (\mathbb{Z}/2\mathbb{Z})^2,\mathbb{Z}\right) \cong \mathbb{Z}/2\mathbb{Z},$$
and hence $H^1(K,\hat{S})$ has order at most $2$, as wished.

Finally, using the same argument given at the beginning of this section, we see that the geometric Brauer group of $Z$ is trivial. We conclude then that
$$\mathrm{Br}(Z)/\mathrm{Br}_0(Z) = \mathrm{Br}_{\mathrm{al}}(Z) = \{0,\gamma\}.$$

\emph{Step 8: Checking that $Z(\mathbb{A}_K)^{\br(Z)}\neq \emptyset$.} Consider the map:
$$\mathrm{ev}_{A,v_0}: \tilde{V}(K_{v_0}) \rightarrow \mathrm{Br}(K_{v_0}).$$
Since the extension $K_{v_0}(\sqrt{a},\sqrt{b})/K_{v_0}$ has degree $4$, $v_0(a)=1$ and $v_0(b)=0$, \cite[Prop.~4.3]{ColliotParimalaSuresh} shows that $\mathrm{Im}(\mathrm{ev}_{A,v_0})=\{ 0,(-a,b)\}=\mathrm{Br}(K_{v_0})[2]$. We deduce that there exists $(\tilde P_v) \in \tilde{V}(\mathbb{A}_K)$ such that $\BM((\tilde P_v),[A])=0$. Hence $\BM(r(\tilde P_v),(r^*)^{-1}([A]))=0$. 

Now recall that the projection $q:W \rightarrow V$ is an $\sl_{n,K}$-torsor. Since $\sl_{n,K}$-torsors over a field are always trivial, we can find an adelic point $(P_v) \in W(\mathbb{A}_K)$ that lifts $(r(\tilde P_v)) \in V(\mathbb{A}_K)$. We then have $\BM((P_v),q^*(r^*)^{-1}([A]))=0$, and hence
$$\BM(p(P_v),\gamma) = \BM(p(P_v),(p^*)^{-1}q^*(r^*)^{-1}([A]))=0.$$
Since $\mathrm{Br}(Z)/\mathrm{Br}_0(Z) = \{0,\gamma\}$ according to Step 7, we deduce that:
$$p(P_v) \in Z(\mathbb{A}_K)^{\mathrm{Br}(Z)}.$$
This finishes the proof of statement (iii) and hence the proof of the theorem.
\end{proof}

\begin{rem}
It is highly likely that there are such counterexamples over \emph{every} field $K$ of type (a) and (b). Indeed, consider a finite extension $L/K$ such that $L$ satisfies the assumptions of Theorem \ref{thm counter}. Consider then the Weil restriction $R_{L/K}(Z)$, where $Z$ is defined as in Theorem \ref{thm counter}. This variety is a homogeneous space under $R_{L/K}(\mathrm{\sl}_{n,L})$. It is evident that this variety satisfies conditions (ii) and (iv). Finally, for (iii), note that $\overline{R_{L/K}(Z)}$ is simply a self-product of $\bar{Z}$ and hence $\br(\overline{R_{L/K}(Z)})=0$, so that
\begin{align*}
\br(R_{L/K}(Z))/\br_0(R_{L/K}(Z))&=\bral(R_{L/K}(Z))=H^1(K,I_{L/K}(\hat{S}))\\&\cong H^1(L,\hat{S})\cong \br(Z)/\br_0(Z).
\end{align*}
One should check then that the Brauer pairing is compatible with these isomorphisms and the bijection between $M$-points of $K$ and $M\otimes_K L$-points of $Z$ for $M/K$. This seems to be a tedious straightforward computation.
\end{rem}

\section{An obstruction using torsors under quasi-trivial tori}

Let $Z$ be a homogeneous space under a linear connected $K$-group $G$ with geometric stabilizer $\bar H$ satisfying \eqref{eq hyp}. In the previous section, we have seen that the usual Brauer-Manin obstruction is not enough to explain the failure of the local-global principle when $Z$ does not satisfy the technical hypothesis from Theorem \ref{thm Htorf inj Gtor}, that is, the injectivity of the natural arrow $\bar H^\torf\to\bar G^\tor$. This is why, in this section, we aim at constructing a stronger obstruction by taking into account more places of $K$ than those in $X^{(1)}$ and applying the Brauer-Manin obstruction to torsors under quasi-trivial tori over $Z$.

For that purpose, we start by defining a new set $\Omega_{Z}$ of places of $K$ as follows. Consider the canonical $K$-form $H^\torf$ of $\bar H^\torf$ associated to $Z$. We introduce the following notations:
\begin{itemize}
\item $L/K$ is the minimal extension splitting the group of multiplicative type $H^\torf$;
\item $B$ is the normalization of $A$ in $L$;
\item $\cal Y$ is an integral regular $2$-dimensional scheme with function field $K$ endowed with a projective surjective morphism $q:\cal Y\to\spec(B)$ such that its special fiber $Y_0$ is a strict normal crossings divisor;
\item $Y$ is the generic fiber of $q$;
\item $\Omega_{Z}$ is the set of places $v$ of $K$ that are induced by a place $w\in\cal Y^{(1)}$;
\item $\Omega_{0,Z}$ is the set of places $v$ of $K$ that are induced by a place $w\in\cal Y^{(1)}$ and that are \emph{not} in $X^{(1)}$.
\end{itemize}
Note that $\Omega_{0,Z}$ may contain places that are not in $X_0^{(0)}=\cal X^{(1)}\smallsetminus X^{(1)}$ and that both $\Omega_Z$ and $\Omega_{0,Z}$ depend on the choice of $\cal Y$.

Finally, we fix an inclusion $H^\torf\hookrightarrow T$ into a torus $T$ isomorphic to $(R_{L/K}\gm)^n$ for some $n \geq 0$. Note that such an inclusion always exists.\\

We start with an application of \cite{DLA} that allows us to construct a torsor over $Z$ under $T$.

\begin{pro}\label{prop geom const}
With notation as above, assume that $Z(K_v)\neq\emptyset$ for every $v\in\Omega_Z$. Then there exists a $K$-homogeneous space $W_{Z,G,\overline{H}}$ under $G\times T$ with geometric stabilizer $\bar H^\torf$ with the following extra properties:
\begin{itemize}
\item the natural homomorphism $\bar H^\torf\to (\bar G\times \bar T)^\tor=\bar G^\tor\times \bar T$ is injective;
\item there is a morphism of $K$-homogeneous spaces $p:W_{Z,G,\overline{H}}\to Z$ which makes $W_{Z,G,\overline{H}}$ into a $T$-torsor.
\end{itemize}
\end{pro}

\begin{proof}
We use the notation of Section \ref{sec preliminaires}. Consider the class $\eta^\torf\in H^2(K,H^\torf)$ naturally associated to $Z$ (see the text after Proposition \ref{prop conn stab}) and denote by $\xi$ its image in $H^2(K,T)\simeq H^2(L,\gm)^n$. For $v\in\Omega_Z$, the hypothesis $Z(K_v)\neq\emptyset$ implies that $\eta^\torf$ is trivial in $H^2(K_v,H^\torf)$. Hence $\xi$ lies in $\Sha^2_{\mathcal{Y}}(L,\gm)^n$, which is trivial by \cite[Rem.~2.3]{ColliotParimalaSuresh}. In particular, $\xi$ represents the trivial gerbe $\underline{\mathrm{TORS}}(T)$. Thus, the fact that $\xi$ is the image of $\eta^\torf$ can be interpreted as a morphism of gerbes $\cal M^\torf_Z\to\underline{\mathrm{TORS}}(T)$. By \cite[Thm.~3.4]{DLA}, we get all the data in the statement of the theorem (using the notations in \cite{DLA}, take $N:=G^\ssu$, $G':=T$ and $\cal M':=\cal M^\torf_Z$).
\end{proof}

\begin{rem}
This kind of construction was already used by Borovoi in \cite{Borovoi96} in order to reduce the study of the Brauer-Manin obstruction to the Hasse principle and weak approximation for homogeneous spaces over number fields to simpler cases in which this study had already been done.
\end{rem}

We now prove that the Brauer-Manin obstruction for this $T$-torsor $W_{Z,G,\overline{H}}\to Z$ is enough to explain the eventual failure of the local-global principle for $Z$.

\begin{thm}\label{thm new obst}
Let $K$ be a field of type (a) or (b). Let $Z$ be a homogeneous space under a connected linear group $G$ with geometric stabilizer $\bar H$ satisfying \eqref{eq hyp}. We keep the notations given at the beginning of this section. In particular, $T$ is a quasi-trivial torus split by $L$. Assume that $Z(K_v)\neq\emptyset$ for every $v\in\Omega_{Z}$, and consider the $T$-torsor $W_{Z,G,\overline{H}} \rightarrow Z$ constructed in Proposition \ref{prop geom const}. Assume moreover that $W_{Z,G,\overline{H}}(\bb A_K)^{\Brusse(W_{Z,G,\overline{H}})}\neq\emptyset$. Then $Z(K)\neq \emptyset$.
\end{thm}

In particular, by Proposition \ref{prop conn stab}, we obtain the result for homogeneous spaces with connected stabilizers stated in Theorem \ref{mt2}.

\begin{proof}
Proposition \ref{prop geom const} gives us the $T$-torsor $W_{Z,G,\overline{H}}\to Z$, to which we apply Theorem \ref{thm Htorf inj Gtor}. We conclude that
\[W_{Z,G,\overline{H}}(\bb A_K)^{\Brusse(W_{Z,G,\overline{H}})}\neq\emptyset\,\Rightarrow\, W_{Z,G,\overline{H}}(K)\neq\emptyset\, \Rightarrow\, Z(K)\neq \emptyset,\]
where the second implication is obvious.
\end{proof}

\begin{rem}
Another way to define this obstruction is as follows. Assuming $Z(K_v)\neq\emptyset$ for every $v\in\Omega_Z$, we obtain a $T$-torsor $W_{Z,G,\overline{H}}\to Z$. Since $T$ is quasi-trivial, every adelic point $(P_v)\in Z(\bb A_K)$ lifts to an adelic point $(Q_v)\in W_{Z,G,\overline{H}}(\bb A_K)$. Evaluation at $(Q_v)$ defines then a homomorphism $\phi_Z:\Brusse(W_{Z,G,\overline{H}})\to\Q/\Z$ which is independent of the choice of $(P_v)$ and $(Q_v)$. By definition, this means that, for any $(Q_v) \in W_{Z,G,\overline{H}}(\bb A_K)$ and any $\alpha \in \Brusse(W_{Z,G,\overline{H}})$, the Brauer-Manin pairing can be computed by the formula $\mathrm{BM}\left( (Q_v),\alpha \right)=\phi_Z (\alpha)$. Hence the condition $W_{Z,G,\overline{H}}(\bb A_K)^{\Brusse(W_{Z,G,\overline{H}})}\neq\emptyset$ is equivalent to $\phi_Z$ being the trivial morphism.
\end{rem}

In order to get a more conceptual result, we introduce the following obstruction to the local-global principle.

\begin{defi}\label{defi qtBrusse}
For an arbitrary $K$-variety $Z$, we define
\[Z(\bb A_K)^{\qt,\Brusse}:=\underset{T\mathrm{quasi-trivial}}{\bigcap_{f:W\xrightarrow{T} Z}}f(W(\bb A_K)^{\Brusse(W)}).\]
\end{defi}

It is easy to see that $Z(K)\subseteq Z(\bb A_K)^{\qt,\Brusse}\subseteq Z(\bb A_K)$.

\begin{cor}[of Theorem \ref{thm new obst}]\label{cor new obst}
Let $K$ be a field of type (a) or (b). Let $Z$ be a homogeneous space under a connected linear group $G$ with geometric stabilizer $\bar H$ satisfying \eqref{eq hyp}. Then
\[\big [Z(\bb A_K)^{\qt,\Brusse}\neq\emptyset\ \,\text{and}\ \, Z(K_v)\neq\emptyset,\,\forall\, v\in\Omega_{0,Z}\big ]\,\Rightarrow\, Z(K)\neq \emptyset.\]
\end{cor}

\begin{proof}
Since $Z(\bb A_K)^{\qt,\Brusse}\neq\emptyset$, the set $Z(K_v)$ is non-empty for each $v\in X^{(1)}=\Omega_{Z}\setminus \Omega_{0,Z}$.  We deduce that $Z(K_v)\neq \emptyset$ for each $v\in \Omega_Z$. In particular, we can construct the $T$-torsor $W_{Z,G,\overline{H}} \rightarrow Z$ of Proposition \ref{prop geom const}. The assumption $Z(\bb A_K)^{\qt,\Brusse}\neq\emptyset$ implies that $W_{Z,G,\overline{H}}(\bb A_K)^{\Brusse(W_{Z,G,\overline{H}})}\neq\emptyset$ and hence, by Theorem \ref{thm new obst}, we get $Z(K) \neq \emptyset$.
\end{proof}

\begin{rem}
The obstruction in Corollary \ref{cor new obst} is clearly not computable. As for the obstruction of Theorem \ref{thm new obst}:
\begin{enumerate}
    \item one can explicitly obtain a model $\cal Y$ of the field $L$ by taking the normalization of $\cal X$ in $L$ and desingularizing;
    \item the set $\Omega_{0,Z}$ of extra places is finite and hence the condition $Z(K_v) \neq \emptyset$ for $v\in \Omega_{0,Z}$ is explicitly computable;
    \item the morphism $H^\torf\to T$ is explicitly computable;
    \item the corresponding $T$-torsor $W_{Z,G,\overline{H}} \rightarrow Z$ is explicitly computable over any extension $M/K$ such that $Z(M)\neq\emptyset$ and can be obtained over $K$ by (finite) Galois descent;
    \item the homogeneous space $W_{Z,G,\overline{H}}$ satisfies the hypotheses of Proposition \ref{prop fact brusse barrre}, hence the group $\overline{\Brusse(W_{Z,G,\overline{H}})}$ is finite;
    \item given an element $\alpha \in \br(W_{Z,G,\overline{H}})$ that lifts an element of $\overline{\Brusse(W_{Z,G,\overline{H}})}$, the value of the Brauer-Manin pairing $\mathrm{BM}((P_v),\alpha)$ does not depend on the adelic point $(P_v)$ and can be explicitly computed.
\end{enumerate}
The only constructions for which we do not know whether they are algorithmically computable are the Galois descent datum for the $T$-torsor $W_{Z,G,\overline{H}}$, the computation of the group $\overline{\Brusse(W_{Z,G,\overline{H}})}$ and the lifts of its elements to $\br(W_{Z,G,\overline{H}})$.
\end{rem}

\section{Comparing our new obstruction with descent obstructions}\label{descent sec}

It is well-known that, in the context of number fields, the Brauer-Manin obstruction with respect to the whole Brauer group is equivalent to the descent obstruction with respect to connected groups (cf.~\cite[Prop.~5.3.4]{Skor} and \cite[Thm.~2]{Har02}). Furthermore, for proper varieties, the Brauer-Manin obstruction with respect to the \emph{algebraic} Brauer group is well-known to be equivalent to the descent obstruction with respect to tori (cf.~\cite[Thm.~6.1.1]{Skor} and \cite[Thm.~2]{Har02}). Now, we have just seen that in this new context the Brauer-Manin obstruction fails to explain some failures of the local-global principle, but that the extra input of torsors under quasi-trivial tori seems to be enough to explain everything. It is natural then to compare this new obstruction with the descent obstruction with respect to tori.

In this section we prove that, up to considering the new set of places $\Omega_Z$, which seems to be unavoidable, the descent obstruction with respect to tori is not weaker than the obstruction introduced in Definition \ref{defi qtBrusse} and hence it also explains the failures of the local-global principle for homogeneous spaces satisfying \eqref{eq hyp}.\\

Let us recall the classical definitions associated to descent obstructions (cf.~\cite[\S5.3]{Skor}).

\begin{defi} Given a torsor $f:W\to Z$ under a $K$-group $G$, we define
\[Z(\bb A_K)^f:=\bigcup_{[a]\in H^1(K,G)} ({}_af)({}_aW(\bb A_K)).\]
For an arbitrary $K$-variety $Z$, we define
\[Z(\bb A_K)^\tor:=\underset{T\,\mathrm{torus}}{\bigcap_{f:W\xrightarrow{T} Z}} Z(\bb A_K)^f.\]
\end{defi}

We give now a generalization of Definition \ref{defi qtBrusse}, which we will compare with descent obstructions here below.

\begin{defi}
For $B\in\{\Brusse, \brun, \br\}$, define
\[Z(\bb A_K)^{\qt,B}:=\underset{T\,\mathrm{quasi-trivial}}{\bigcap_{f:W\xrightarrow{T} Z}}\bigcap_{\alpha\in B(W)} f(W(\bb A_K)^\alpha).\]
\end{defi}

\begin{rem}
Recall that quasi-trivial tori have trivial $H^1$ by Hilbert's Theorem 90 and hence there is no need to consider Galois twists in the last definition. In particular, for $B=\Brusse$ we recover Definition \ref{defi qtBrusse} since $W(\bb A_K)^\alpha$ is either empty or the whole set $W(\bb A_K)$ and thus
\[\bigcap_{\alpha\in \Brusse(W)} f(W(\bb A_K)^\alpha)=f(W(\bb A_K)^{\Brusse(W)}).\]
Note, however, that for $B=\brun$ or $B=\br$ this definition \emph{does not} coincide \textit{a priori} with the ``more natural'' set
\[\underset{T\mathrm{quasi-trivial}}{\bigcap_{f:W\xrightarrow{T} Z}}f(W(\bb A_K)^{B(W)}).\]
\end{rem}

We can state now the main result of this section.

\begin{thm}\label{thm tor vs qtBr}
Let $K$ be a field of type (a) or (b). Let $Z$ be a smooth geometrically integral $K$-variety. We have $Z(\bb A_K)^\tor\subseteq Z(\bb A_K)^{\qt,\brun}$.
\end{thm}

We then deduce that the descent obstruction with respect to tori, and the existence of local points at the places in $\Omega_{0,Z}$, are enough to explain the lack of Hasse principle for the homogeneous spaces considered in this article:

\begin{cor}\label{descent hom}
Let $K$ be a field of type (a) or (b). Let $Z$ be a homogeneous space under a connected linear group $G$ with geometric stabilizer $\bar H$ satisfying \eqref{eq hyp}. Then
\[\big [Z(\bb A_K)^{\tor}\neq\emptyset\ \,\text{and}\ \, Z(K_v)\neq\emptyset,\,\forall\, v\in\Omega_{0,Z}\big]\,\Rightarrow\, Z(K)\neq \emptyset.\]
\end{cor}

\begin{proof}
If $Z(\bb A_K)^{\tor}\neq\emptyset$, then by Theorem \ref{thm tor vs qtBr} the sets $Z(\bb A_K)^{\qt,\brun}\subset Z(\bb A_K)^{\qt,\Brusse}$ are non-empty, and we conclude by Corollary \ref{cor new obst}.
\end{proof}

\begin{proof}[Proof of Theorem \ref{thm tor vs qtBr}]
Consider an adelic point $(P_v)\in Z(\bb A_K)^\tor$. We must prove that, for every quasi-trivial torus $T$, for every $T$-torsor $W\to Z$ and every class $\alpha\in\brun(W)$, we can find a lift of $(P_v)$ to $W(\bb A_K)$ that is orthogonal to $\alpha$. Since the Brauer group coincides with the Azumaya Brauer group, we may follow the proof of \cite[Prop.~5.3.4]{Skor}, adapting it to our context. Note that this proof uses the bijectivity of the map
\[H^1(L,\pgln)\xrightarrow{d_n} \br(L)[n],\]
for all the involved local and global fields, and this holds for a given field $L$ if and only if it has the ``period = index'' property. Since this is well-known for the fields $L$ considered here (cf.~Proposition \ref{prop coh dim 2 and period-index}) we conclude that, whenever we are given a class $\alpha\in\brun(W)$ and a $\pgln$-torsor $f:V\to W$ representing $\alpha$, we have $W(\bb A_K)^\alpha=W(\bb A_K)^f$. Note that the torsor is trivialized over $\bar K$ since $\alpha\in\brun(W)$.

Fix then such a tower $V\to W\to Z$. Let $L/K$ be an extension such that $\alpha|_L=0\in\br(W_L)$. Set $S:=R_{L/K}\gm$ and consider the natural inclusion $i:\gm \rightarrow S$. The injective morphism:
\begin{align*}
    \mathbb{G}_{\mathrm{m}} & \mapsto S \times \gln \\
    t & \mapsto (i(t),t^{-1}\mathrm{Id}_n)
\end{align*}
allows one to regard $\gm$ as a central subgroup of $S \times \gln$. One then easily checks that the quotient $G:=(S \times \gln)/\gm$
fits into the following commutative diagram with exact rows:
\[\xymatrix{
1 \ar[r] & \gm \ar[r] \ar[d]^{i} & \gln \ar[d] \ar[r] & \pgln \ar@{=}[d] \ar[r] & 1\\
1 \ar[r] & S \ar[r] & G \ar[r] & \pgln \ar[r] & 1.
}\]
This diagram gives rise to the following diagram with exact rows in non-abelian cohomology:
\[\xymatrix{
H^1(W,\gln) \ar[r] \ar[d] & H^1(W,\pgln) \ar[r]^{\Delta} \ar@{=}[d] & H^2(W,\gm) \ar[d] \\
H^1(W,G) \ar[r] & H^1(W,\pgln) \ar[r]^{\Delta} & H^2(W,S). \\
}\]
Note that the top $\Delta$-arrow sends the class of $[V\to W]$ to $\alpha$ and that the right-hand vertical arrow corresponds to the restriction map $\br(W)\to \br(W_L)$, hence the image of $\alpha$ in $H^2(W,S)$ is trivial. This implies that the class $[V\to W]$ comes from $H^1(W,G)$. In other words, there exists a $G$-torsor $U\to W$ dominating $V\to W$.

Note now that $\sln$ is normal in $G$ since $S$ is central and they both generate $G$. In particular, $G/\sln$ is a torus $R$. Having said this, we get the following diagram of torsors:
\begin{equation}\label{diag torsor tower}
\xymatrix@=1em{
& U \ar[dl]_S \ar[dd]_G \ar[dr]^{\sln} \\
V \ar[dr]_\pgln && Y \ar[dl]^{R} \\
& W \ar[d]_T \\
& Z.
}
\end{equation}
By Theorem \ref{thm empilement} in the appendix, we know that there exists a $K$-torus $Q$ fitting into an extension
\[1\to R \to Q \to T\to 1,\]
such that the arrow $Y\to Z$ has the structure of a $Q$-torsor. Since $(P_v)\in Z(\bb A_k)^\tor$, we know that there exists a Galois twist ${}_aY$ of $Y$ for $[a]\in H^1(K,Q)$ such that $(P_v)$ lifts to an adelic point in ${}_aY$. Now, since $T$ is quasi-trivial, we immediately see that $[a]$ comes from a class in $H^1(K,R)$, abusively still denoted by $[a]$. Now, consider the exact sequence 
\[1\to \sln\to G\to R\to 1,\]
and the corresponding exact sequence in non-abelian cohomology (cf.~\cite[Rem.~IV.4.2.10]{Giraud}):
\[ H^1(K,G) \to H^1(K,R) \xrightarrow{\Delta} H^2(K,\sln).\]
Since $H^2(K,\sln)$ is only composed of neutral classes by Lemma \ref{lem GA}, the class $[a]\in H^1(K,R)$ comes from a class $[b]\in H^1(K,G)$, whose image in $H^1(K,\pgln)$ we denote by $[c]$. We may then twist the whole diagram \eqref{diag torsor tower} by the cocycle $b$ in order to get the following diagram of torsors:
\[\xymatrix@=1em{
& {}_bU \ar[dl]_S \ar[dd]_{{}_bG} \ar[dr]^{\sl(A)} \\
{}_cV \ar[dr]_{\mathrm{PGL}(A)} && {}_aY \ar[dl]^{R} \\
& W \ar[d]_T \\
& Z,
}\]
where $A$ denotes some central simple algebra over $K$ (here $\sl(A)$ and $\mathrm{PGL}(A)$ are inner twists of $\sln$ and $\pgln$). We know then that $(P_v)$ lifts to an adelic point $(P'_v)\in {}_aY(\bb A_K)$. Each $P'_v$ lifts in turn to a $K_v$-point of ${}_bU$ since $H^1(K_v,\sl(A))=1$ for these fields (cf. \cite{Suslin85}). It is easy to see that these lifts define an adelic point in ${}_bU$, which we may push down to an adelic point $P''$ in ${}_cV$ lifting $(P_v)$. Since ${}_cV$ is a Galois twist of $V$, the image of $P''$ in $W$ belongs to $W(\bb A_K)^f$ and lifts $(P_v)$. This concludes the proof.
\end{proof}

\begin{appendix}

\section{A result on towers of torsors}\label{sec appendix}

The goal of this appendix is to prove the following result, which is needed in the proof of Theorem \ref{thm tor vs qtBr} above. We are greatly indebted to Mathieu Florence for his help with the proof.

\begin{thm}\label{thm empilement}
Let $K$ be a field of characteristic 0. Let $G$ be a connected linear $K$-group and $T$ an algebraic $K$-torus. Let $Y\to X$ be a $G$-torsor and let $Z\to Y$ be a $T$-torsor. Assume that $\pic(\bar G)=0$ (which holds, for instance, if $G$ is a torus) and that $X$ is geometrically integral. Then there exists a canonical extension
\[1\to T\to E\to G\to 1,\]
such that the composite $Z\to X$ is an $E$-torsor. Moreover, if $G$ is a torus, then $E$ is a torus as well.
\end{thm}

\begin{rem}
One can find similar results in \cite[App.~A]{McFaddin et cie} and \cite[Lem.~2.13]{BD}, but none of these seems to be general enough for our purposes. We hope to generalize this result even further in the future.
\end{rem}

\begin{proof}
For a $K$-scheme $W/K$, we denote by $X_W,Y_W,Z_W,T_W,G_W$ the $W$-(group-)schemes obtained by base change from $X,Y,Z,T,G$ respectively. Consider the group $\aut_{X_W}^{T_W}(Z_W)$ of $X_W$-auto\-morphisms $\varphi$ of $Z_W$ that are compatible with the action of $T_W$ in the sense that the following diagram commutes:
\[\xymatrix{
Z_W\times T_W \ar[r]^-{a_W} \ar[d]_{\varphi\times\id} & Z_W \ar[d]^{\varphi} \\
Z_W\times T_W \ar[r]^-{a_W} & Z_W,
}\]
where $a$ denotes the morphism defining the action of $T$ on $Z$ and $a_W$ the corresponding morphism after base change. The functor $W/K\mapsto \aut_{X_W}^{T_W}(Z_W)$ defines a group presheaf over the big \'etale site over $K$. Denote by $\underline{\aut}_X^T(Z)$ the corresponding sheaf and consider the subsheaf $\underline{\aut}_{Y}^{T}(Z)$ defined by taking the subgroup $\aut_{Y_W}^{T_W}(Z_W)$ of $\aut_{X_W}^{T_W}(Z_W)$ for each $W/K$. We have
\[\aut_{Y_W}^{T_W}(Z_W)=T(Y_W).\]
Indeed, it is well-known that the functor $W/Y\mapsto \aut_{W}^{T_W}(Z\times_Y W)$ over the big \'etale site of $Y$ is represented, as a $Y$-scheme, by $T_Y$ (cf.~for instance \cite[III.\S1.5]{Giraud}). Moreover, a direct application of Rosenlicht's Lemma gives us, for geometrically irreducible $W$,
\[T(Y_W)=T(W)\times M(W),\]
where $M$ is a free and finitely generated abelian constant sheaf. We deduce then that the quotient sheaf $M$ associated to the presheaf
\[W\mapsto T(Y_W)/T(W),\]
is a locally free, locally constant sheaf that is finitely generated and abelian. In other words, we have an exact sequence of abelian sheaves
\[1\to T\to \underline{\aut}_{Y}^{T}(Z)\to M\to 1.\]
In particular, since $M$ is locally constant, it is representable. And since $T$ is affine, we get by \cite[III.4, Prop.~1.9]{DG} that $\underline{\aut}_{Y}^{T}(Z)$ is represented by an abelian $K$-group scheme $A$. We abusively denote by $A$ the sheaf $\underline{\aut}_{Y}^{T}(Z)$ as well.\\

Since every element in $\aut_{X_W}^{T_W}(Z_W)$ induces an $X_W$-automorphism of $Y_W$, we have an exact sequence of sheaves
\[1\to A\to \underline{\aut}_X^T(Z)\xrightarrow{\pi} \underline{\aut}_X(Y),\]
where $\underline{\aut}_X(Y)$ denotes the sheaf of $X$-automorphisms of $Y$.

We claim that $\pi$ is surjective. In order to prove this, we can replace $K$ by a finite extension. We may assume then that both $T$ and $G$ are split over $K$. Since $G$ is split and connected and $\pic(\bar G)=0$, we have that $\pic(G)=0$ (cf.~\cite[Lem.~6.9]{Sansuc81}). By \cite[Prop.~6.10]{Sansuc81}, we get then that the map $\pic(X)\to\pic(Y)$ is surjective. Since $T$ is split, we deduce then that the torsor $Z\to Y$ comes by pullback from a $T$-torsor $Z'\to X$. In other words, $Z=Y\times_X Z'$. The surjectivity is then evident.

Note now that $G$ is clearly a subgroup of $\underline{\aut}_X(Y)$. Define then $\cal E'\subset \underline{\aut}_X^T(Z)$ to be the group sheaf corresponding to the preimage of $G$ via $\pi$. We get an exact sequence
\[1\to A\to \cal E'\to G\to 1.\]
Since $A$ is abelian, the extension $\cal E'$ induces an action of the \emph{sheaf} $G$ on the \emph{sheaf} $A$. By Yoneda's Lemma, this action is actually an action of the $K$-group $G$ on the $K$-group $A$. Note that $T$ corresponds to the neutral connected component of $A$ and thus it is preserved by the $G$-action since $G$ is connected. In particular, we may quotient by $T$ in order to get an exact sequence
\[1\to M\to\cal F\to G\to 1.\]
Then, if one forgets its group structure, $\cal F$ corresponds to an $M$-torsor over the scheme $G$. By \cite[Exp.~8, Prop.~5.1]{SGA7}, we know that $H^1(G,\bb Z)=0$ and hence, since $M$ is locally free and finitely generated, $H^1(G,M)=0$ up to taking a finite extension of $K$. Since representability can be checked over a finite extension (once again by Yoneda's Lemma), this tells us that $\cal F$ is represented by a $K$-scheme $F$ (which is a $K$-form of $M\times G$). Thus we have an exact sequence of $K$-group-\emph{schemes}
\[1\to M\to F\to G\to 1,\]
from where we get a new exact sequence of group sheaves
\[1\to T\to \cal E'\to F\to 1.\]
Once again, since $T$ is an affine group-scheme, \cite[III.4, Prop.~1.9]{DG} tells us that $\cal E'$ is in fact a scheme $E'$ and hence a $K$-group. But now, since $M$ is discrete, it suffices to define $E$ as the identity component of $E'$, which clearly fits into an exact sequence
\[1\to T\to E\to G\to 1.\]
Since $E\subset E'$ is a subgroup of $\underline{\aut}_X^T(Z)$, it is immediate then to check that $E$ acts on $Z$ and that $Z\to X$ is an $E$-torsor.

Note finally that the whole construction is clearly canonical. Indeed, the extension $\cal E'$ is obtained by pullback from the canonical extension of sheaves of automorphism groups. Then we get $E'$ as \emph{the} $K$-group scheme representing $\cal E'$ and finally we take the identity component, which is a characteristic subgroup. Also, the last assertion when $G$ is a torus follows from \cite[IV.1, Prop.~4.5]{DG}.
\end{proof}

\end{appendix}

\end{document}